\documentclass[11pt]{article}
\usepackage{amsmath}
\usepackage{amsthm}
\usepackage{amssymb}
\usepackage{amscd}
\usepackage[dvipdfm]{hyperref}

\newcommand{\C}{\mathbb{C}}
\newcommand{\R}{\mathbb{R}}

\theoremstyle{plain}
\newtheorem{theorem}{Theorem}[section]
\newtheorem{proposition}[theorem]{Proposition}
\newtheorem{lemma}[theorem]{Lemma}
\newtheorem{Corollary}[theorem]{Corollary}

\theoremstyle{definition}
\newtheorem{definition}[theorem]{Definition}

\makeatletter
\@addtoreset{equation}{section}
\makeatother
\newcounter{constantLABEL}
\newcommand{\cref}[1]{C_{\ref{#1}}}

\newcounter{constantslabel}
\begin{document}


\title{
A removal singularity theorem of the 
 Donaldson--Thomas instanton on compact K\"{a}hler threefolds 
}
 

\author{Yuuji Tanaka}
\date{}


\maketitle


\begin{abstract}
We consider a perturbed Hermitian--Einstein equation, 
which we call the Donaldson--Thomas equation, on compact
 K\"{a}hler threefolds. 
In \cite{tanaka2}, we analysed some analytic properties of solutions  to the
 equation, in particular, we proved that a sequence of solutions to the
 Donaldson--Thomas equation has a subsequence which smoothly converges
 to a solution to the Donaldson--Thomas equation outside a closed 
 subset of the Hausdorff dimension two.  
In this article, we prove that some of  these singularities can be removed. 
\end{abstract}




\markboth{}
{A removal singularity theorem for Donaldson--Thomas instantons}


\section{Introduction}

In \cite{tanaka1}, we introduced a perturbed Hermitian--Einstein
equation on symplectic 6-manifolds  
in order to analytically approach the Donaldson--Thomas invariants 
developed in \cite{Thomas00}, \cite{JS}, \cite{KS}, and \cite{KS2}. 
In \cite{tanaka1}, we described the local structures of the moduli space of
the Donaldson--Thomas instantons, 
and the moment map description of the moduli space. 
Subsequently, in \cite{tanaka2}, we proved a weak convergence
theorem of the Donaldson--Thomas instantons on compact K\"{a}hler
threefolds. 
This article is a sequel of \cite{tanaka2}.   
We prove that some of singularities 
which appeared in \cite{tanaka2} are removable.

Firstly, let us introduce the equations.    
Let $Z$ be a compact K\"{a}hler threefold with K\"{a}hler form
$\omega$, and let $E$ be a unitary vector bundle over $Z$ of rank $r$. 
A complex structure on $Z$ gives the splitting of 
the space of the complexified two forms as $
 \Lambda^2 \otimes \C = \Lambda^{1,1} \oplus \Lambda^{2,0} 
\oplus \Lambda^{0,2}$, 
and $\Lambda^{1,1}$ further decomposes into $ \C \langle \omega \rangle \oplus \Lambda_{0}^{1,1}$. 
We consider the following equations for a connection $A$ of $E$, 
 and an $\mathfrak{u} (E)$-valued (0,3)-form $u$ on $Z$.  
\begin{gather}
F_{A}^{0,2} =0 , \quad   \bar{\partial}_{A}^{*} u = 0, 
\label{DT1}
\\
F_{A}^{1 ,1} \wedge \omega^{2} + [ u , \bar{u}] 
+  i \frac{\lambda(E)}{3} Id_{E} \omega^3 = 0 ,
\label{DT2}
\end{gather}
where $\lambda (E)$ is a constant defined by 
$\lambda (E) := 
6 \pi ( c_1(E) \cdot [\omega]^{2} ) / r [\omega]^{3}$. 
We call these equations the {\it Donaldson--Thomas
equations}, and a solution $(A, u)$ to these equations a  
{\it Donaldson--Thomas
instanton} (or a {\it D--T instanton} for short).

In \cite{tanaka2}, 
we proved the following weak convergence theorem for the
Donaldson--Thomas instantons on compact K\"{a}hler threefolds.

\begin{theorem}[\cite{tanaka2}]
Let $Z$ be a compact K\"{a}hler threefold, and let $E$ be a 
unitary vector bundle over
 $Z$.  
Let $\{ (A_n , u_n) \}$ be a sequence of D--T instantons of E. 
We assume that  $\int_{Z} |  u_{n} |^2 dV_{g}$ are uniformly bounded. 
Then there exist a subsequence $\{ (A_{n_{j}} , u_{n_{j}}) \}$ of $\{ (A_n , u_n) \}$,  
a closed subset $S$ of $Z$ whose real two-dimensional Hausdorff
 measure is finite, and a sequence of gauge transformations 
$\{ \sigma_{j} \}$over $Z \setminus S$
 such that 
$\{ \sigma_{j}^{*} (A_{n_{j}} , u_{n_{j}}) \}$ smoothly converges to a
 D--T instanton over 
$Z \setminus S$. 
\label{th:weak}
\end{theorem}

This leads to introducing the following notion of an \textit{admissible D--T
instanton}, as in \cite{Tian00}. 
\begin{definition}
A smooth D--T instanton $(A,u)$ defined outside a closed subset 
$S$ in $Z$ is called an {\it admissible D--T instanton} if 
$\mathcal{H}^2 (S \cap K) < \infty$ for any compact subset $K \subset Z$,  
and $\int_{Z \setminus S} |F_{A}|^2 \, dV_{g} < \infty$.  
\end{definition}

In this article, we prove the following.  
\begin{theorem}
Let $B_{r} (0) \subset \C^3$ be a ball of radius $r$ centred at the origin with
 K\"{a}hler metric $g$. We assume that the metric is compatible with the
 standard metric by a constant $\Lambda$. 
Let $E$ be a vector bundle over $B_{r} (0)$,  
and let $(A, u)$ be an admissible D--T instanton of $E$  
with $\int_{B_{r} (0)} 
|u|^2 dV < \infty $. 
Then, there exists a constant $\varepsilon > 0$ such that 
if $ \frac{1}{r^2} \int_{B_{r}(0)} |F_{A}|^2  dV_{g} \leq \varepsilon$, 
there exists a smooth gauge transformation $\sigma$ on 
$B_{\frac{r}{2}}(0) \setminus S$ such that $\sigma (A,u)$ 
smoothly extends over $B_{\frac{r}{2}} (0)$
\label{ith:remov}
\end{theorem}

Theorem \ref{ith:remov} implies that the ``top stratum'' of the
singular set $S$, which we denote by $S^{(2)}$ (see Section 2 for
its definition), can be removed. 
Our argument goes through in a similar way to that of Tian--Yang
\cite{TY02}, and we follow Nakajima \cite{Nakajima87} and 
Uhlenbeck \cite{Uh1} for removing
isolated singularities, 
except that we deal with an additional nonlinear term coming from the
extra 
field $u$. 
In Section 2, we bring some results from \cite{tanaka2}, which will 
be used in this article, and describe some structures of the singular sets. 
We then prove Theorem \ref{ith:remov} in Section 3.

\paragraph{Notation.} Throughout this article, 
$C, C'$, and $C''$ are positive constants, but they can be different each
time they 
occur.


\section{Weak convergence and the singularities}

In this section, we bring results from \cite{tanaka2}, 
and describe some properties of the singular set in Theorem \ref{th:weak}.

Firstly, we recall the following
monotonicity formula for the Donaldson--Thomas instantons on compact
K\"{a}hler threefolds.

\begin{proposition}[\cite{tanaka2}]
Let $(A,u)$ be a D--T instanton of 
a unitary vector bundle $E$ over a compact K\"{a}hler threefold $Z$.  
Then, for any $z \in Z$, there exists a  positive constant $r_z$ such that 
for any $0< \sigma < \rho< r_z$, the following holds.  
\begin{equation}
\begin{split}
&\frac{1}{\rho^2} e^{ a \rho^2} \int_{B_{\rho}(z)} 
 m(A,u) dV_g 
 -  \frac{1}{\sigma^2} e^{
a \sigma^2} \int_{B_{\sigma}(z)} 
 m (A,u) dV_g \\
& \quad \qquad \geq   \int_{\sigma}^{\rho} 8 \tau^{-3} e^{a \tau^2}
  \int_{B_{\tau} (z)}  | [ u, \bar{u}]|^2 dV_{g} d \tau \\ 
  &\qquad \qquad +  \int_{B_{\rho} (z) \setminus B_{\sigma} (z)} r^{-2} e^{a
  r^2}  \left\{  4 \left|  \frac{\partial }{\partial r}  
 \lfloor F_{A}^{\perp} \right|^2 
 - 12 \left| \frac{\partial }{\partial r}  
 \lfloor [u , \bar{u}] \right|^2 \right\}  dV_{g} , \\ 
\label{eq:monotone}
\end{split}
\end{equation}
where 
${m} (A, u) := |F_{A}^{\perp}|^2 -  | [u, \bar{u} ]|^2$, and 
$a$ is a constant which depends only on $Z$.
\label{th:monotone}
\end{proposition}

Another advantage of working on compact K\"{a}hler threefolds is the
following estimate on the extra field $u$: 
\begin{proposition}[\cite{tanaka2}] 
Let $(A,u)$ be a D--T instanton 
of 
a unitary vector bundle $E$ over a compact K\"{a}hler threefold $Z$. 
Then we have 
\begin{equation}
 || u ||_{L^{\infty}} 
\leq C || u ||_{L^2} , 
\end{equation}
where $C >0 $ is a positive constant which depends only on $Z$. 
\label{estu}
\end{proposition}

In this article, we use the following curvature estimate essentially
proved in \cite{tanaka2} by using the monotonicity formula and the
estimate on $u$ above. 
\begin{proposition}
Let $(A,u)$ be a D--T instanton 
of 
a unitary vector bundle $E$ over a compact K\"{a}hler threefold $Z$.  
Then there exist constants $\varepsilon >0$ 
and $C_1 > 0 $
 which depend only on $Z$ such that 
for any $z \in Z$ and $0 < r < r_z$, where $r_z$ is the constant in
 Proposition \ref{th:monotone},  if 
$ \frac{1}{r^{2}} \int_{B_{r}(z)} |F_{A}|^2 \, dV_g \leq \varepsilon$ 
and $\int_{B_{r}(z)} |u|^2 dV_{g} < \varepsilon$, 
then 
$$ | F_{A} | (z) 
\leq \frac{C_1}{r^2} \left( \frac{1}{r^{2}} 
\int_{B_{r}(z)} |F_{A}|^2 \,  dV_g  \right)^{\frac{1}{2}} 
+ C_1 \varepsilon r. 
$$
\label{th:est}
\end{proposition}

With these above in mind, we next recall that 
the singular set $S$ in Theorem \ref{th:weak} is given by 
$$
S := {S} (\{ (A_n , u_n )\})
:= \bigcap_{r>0} \{ z \in Z \, : \, \liminf_{n\to\infty}  
\frac{1}{r^2} \int_{B_r (z)}
m (A_{n} , u_{n})  \, dV_g \geq \varepsilon\}. 
$$

We note that this set $S$ is 2-rectifiable, which can be proved by using  
a result by Preiss \cite{Preiss} (see also \cite{Mos}) as follows.   
Let $\mu_{n} := m (A_n , u_n ) dV_{g}$ be a sequence of Radon measures. 
Then, we may assume that 
$\mu_n$ converges to a Radon measure $\mu$ on $Z$ 
after taking a subsequence if necessary, that is, 
for any continuous function $\phi$ 
of compact support over $Z$, 
$ \lim_{n \to \infty} \int_{Z} \phi \,  m (A_n , u_n ) 
dV_g = \int_{Z} \phi
d \mu $. 
We can write 
$ \mu = m (A , u ) dV_g + \nu $, 
where $\nu$ is a non-negative Radon measure on $Z$. 
From Proposition \ref{th:monotone}, for any $z \in Z$, 
$e^{ar^2} r^{-2} \mu (B_{r} (z))$ is a non-decreasing function of $r$, 
hence the density $\Theta (\mu ,z) := \lim_{r \to 0 +} r^{-2} \mu (B_{r}
(z))$ exists for all $z \in Z$, and from the definition of $S$, $z \in S$ if and
only if $\Theta (\mu ,z) \geq \varepsilon$. 
Hence, for $\mathcal{H}^{2}\text{-a.e.}\, \, \, \,  z \in S ( \{
 (A_n , u_n )  \})$, we can write  
$ \nu (z) = \Theta (\mu , z) \mathcal{H}^{2} \lfloor S (\{  (A_n , u_n )
\})$. 
Note that we have $\Theta (\mu, z) 
\leq 4^{2} r_{z}^{- 2} e^{a r_{z}^2} C$  
from the monotonicity formula.  
We also have the following: 
\begin{proposition}
$\mathcal{H}^2$-a.e. $z \in S$. 
$$ \lim_{r \to 0+} 
\frac{1}{r^2} \int_{B_{r}(z)} |F_{A}|^2 \, dV_g =0 .$$
\label{prop:sz}
\end{proposition}

\begin{proof}
We follow an argument by Tian \cite[pp.~222]{Tian00} 
(see also \cite[\S 1.7 Cor.~3]{sim}).   
We consider 
$ S_{j} 
:= \{ z \, : \varlimsup_{r \to 0+} \int_{B_{r} (z)} 
|F_{A} |^2 \, dV_{g} > j^{-1} \} $, 
and prove that $\mathcal{H}^{2} (S_{j}) =0 $ for each $j \geq 1$.

For $\delta > 0$,
we take a covering $\{ B_{\delta} (z_{\alpha}) \}_{\alpha = 1, \dots , N}$ 
of $S_{j}$ such that $z_{\alpha} \in S_{j} \, (\alpha = 1, \dots , N) $,  
and $B_{\delta /2} (z_{\alpha}) 
\cap B_{\delta /2} (z_{\beta}) = \emptyset$ if $\alpha \neq \beta$. 
By definition, for any $z \in S_{j}$, we have 
$\frac{1}{r^2} \int_{B_{r} (z)} |F_{A}|^2 dV_{g} > j^{-1}$ for $r < \delta$.  
Thus, we have  
\begin{equation} 
N \left( \frac{\delta}{2} \right)^2 \leq 
 j \int_{\bigcup B_{\delta /2 (z_{\alpha})}} 
 |F_{A}|^2 \, dV_{g} \leq j 
\int_{S_{j}^{\delta}} |F_{A}|^2 \, dV_{g} , 
\label{eq:sz}
\end{equation}
where $S_{j}^{\delta} = \{ z \in Z \, : \, 
\text{dist} (z , S_{j} ) < \delta \}$. 
Hence, we get 
\begin{equation}
 N \delta^{6} 
\leq 2^{2} \delta^4 j
\int_{S_{j}^{\delta}} |F_{A}|^2 \, dV_{g} . 
\label{sz} 
\end{equation}
Since $B_{\delta} (z_{\alpha})$'s cover $S_{j}$, 
\eqref{sz} implies that 
$\mathcal{H}^{6} (S_{j}) \leq C \delta^4 \int_{S_{j}^{\delta}} 
 | F_{A} |^2 \, dV_{g} $. 
Thus,  
as we can take $\delta \downarrow 0$, 
$S_{j}$
 has the Lebesgue measure zero. 
Then by the dominated convergence theorem, 
we get $\lim_{\delta \to 0} \int_{S_{j}^{\delta}} |F_{A}|^2 \to 0$. 
Therefore, by \eqref{eq:sz}, we conclude 
$\mathcal{H}^{2} (S_{j}) =0$. 
\end{proof}

In order to see that the singular set $S$ is 2-rectifiable, 
we invoke the following theorem by Preiss.  
\begin{theorem}[Preiss \cite{Preiss}, see also \cite{Mat}, \cite{DeL}]
If $0 \leq m \leq p$ are integers, $\Omega$ is a Borel measure on $\R^p$
 such that 
$0 < \lim_{r \to 0} \frac{\Omega (B_{r} x)}{r^{m}} < \infty$ for almost
 all $x \in \Omega$, then $\Omega$ is $m$-rectifiable. 
\label{th:preiss}
\end{theorem}

Theorem \ref{th:preiss} 
tells us that 
the singular set $S$ 
of a weak convergence sequence $\{ (A_n , u_n ) \}$ 
in Theorem \ref{th:weak} 
is 2-rectifiable. In particular, 
for  $\mathcal{H}^{2}$-a.e. $s \in S$
there exists a unique tangent plane $T_s S$.

We remark few more on the structure of the singular sets in Theorem
\ref{th:weak}. 
We put 
\begin{equation}
S^{(2)} :
= \overline {\{z\in S(\{  (A_n , u_n )  \})\,|\, \Theta(\mu, z) > 0,~
\lim_{r\to 0+} \frac{1}{r^{2}} \int_{B_r(z)} |F_{A}|^2  \, 
dV_g=0\}}. 
\label{S2}
\end{equation}
As we describe in Section 3, the limit solution $(A,u)$ extends across
this set $S^{(2)}$. 
We define $S^{(0)} = S \setminus S^{(2)}$, namely, 
$$ S^{(0)} 
= 
{\{z \in S(\{  (A_n , u_n )  \})\,|\, \Theta(\mu, z) > 0,~
\lim_{r\to 0+} \frac{1}{r^{2}} \int_{B_r(z)} |F_{A}|^2 \, 
dV_g \geq \varepsilon \}}. $$
This 
$S^{(0)}$ may be seen as the set of ``unremovable'' singularities, 
however, 
we have the following for the size of this $S^{(0)}$   
from Proposition \ref{prop:sz} at the moment. 
\begin{Corollary}
$$ \mathcal{H}^{2} (S^{(0)}) =0 . $$
\end{Corollary}

We further expect that $\mathcal{H}^{0} (S^{(0)}) < \infty$, and it is
discrete. 
This will be discussed somewhere else.

\section{Removable singularity}

In this section, we prove the following 
removal singularity theorem for D--T
instantons.

\begin{theorem}[Theorem 1.1]
Let $B_{r} (0) \subset \C^3$ be a ball of radius $r$ centred at the origin with
 K\"{a}hler metric $g$. We assume that the metric is compatible with the
 standard metric by a constant $\Lambda$. 
Let $E$ be a vector bundle over $B_{r} (0)$,  
and let $(A, u)$ be an admissible D--T instanton of $E$  
with $\int_{B_{r} (0)} 
|u|^2 dV < \infty $.   
Then, there exists a constant $\varepsilon > 0$ such that 
if $ \frac{1}{r^2} \int_{B_{r}(0)} |F_{A}|^2  dV_{g} \leq \varepsilon$, 
there exists a smooth gauge transformation $\sigma$ on 
$B_{\frac{r}{2}}(0) \setminus S$ such that $\sigma (A,u)$ 
smoothly extends over $B_{\frac{r}{2}} (0)$
\label{th:remov}
\end{theorem}

An immediate corollary of this is the following. 
\begin{Corollary}
Let $(A,u)$ be the limit solution in Theorem \ref{th:weak}, 
and let $S^{(2)}$ be the top stratum of the singular set $S$ defined by \eqref{S2}. 
Then $(A,u)$ extends smoothly across $S^{(2)}$. 
\end{Corollary}

\paragraph{Proof of Theorem \ref{th:remov}.}
First, we recall 
the following result by Bando and Siu.  
\begin{theorem}[\cite{BS94}]
 Let $E$ be a holomorphic vector bundle with Hermitian metric $h$ over 
a K\"{a}hler manifold $Z$ (not necessarily compact
 nor complete) outside a closed subset $S$ of $Z$ with locally finite Hausdorff
 measure of real codimension four. 
We assume that its curvature tensor $F$
 is locally square integrable on $Z$.  Then 
\begin{enumerate}
 \item[$(a)$] $E$ extends to the whole $Z$ as a reflexive sheaf
	    $\mathcal{E}$, and for any local section $s \in \Gamma (U ,\mathcal{E})$, 
$\text{log}^{+} h (s,s)$ belongs to $H^{1}_{loc}$. 
\item[$(b)$] If $\Lambda F$ is locally bounded, then $h (s,s)$ is
	      locally bounded, and $h$ belongs to $L^{p}_{2, loc}$ for
	      any finite $p$ where $\mathcal{E}$ is locally-free. 
\end{enumerate}
\label{th:BS}
\end{theorem}

Combining Theorem \ref{th:BS} with standard elliptic theory, 
we deduce that 
the limiting D--T instanton in Theorem \ref{th:weak} extends 
smoothly on the locally-free part of a reflexive sheaf over $Z$. 
Thus we consider removing isolated singularities.

As the condition $r^{-2} \int_{B_{r} (0)} | F_{A} |^2 \, dV \leq
\varepsilon$ is scale-invariant, we  
assume that $r=1$ below. 
We prove the following in the rest of this section. 
\begin{proposition}
Let $B \subset \C^3$ be the unit ball centred at the origin with
 K\"{a}hler metric $g$. We assume that the metric is compatible with the
 standard metric by a constant $\Lambda$. 
Let $E$ be a vector bundle over $B \setminus \{ 0 \}$,  
and let $(A, u)$ be a D--T instanton of $E$  
with $\int_{B} 
|u|^2 dV < \infty $. 
Then, there exists a constant $\varepsilon > 0$ such that 
if $\int_{B} |F_{A}|^2  dV_{g} \leq  \varepsilon $, 
there exists a smooth gauge transformation $\sigma$ on $B \setminus \{ 0 \}$
 such that 
$\sigma (A,u)$ smoothly extends to a
 smooth D--T instanton over $B$.  
\end{proposition}

\begin{proof}
We follow  a proof by Nakajima \cite{Nakajima87} 
(see also Uhlenbeck \cite{Uh1}) for
 Yang--Mills connections,  except 
that we deal with an additional nonlinear
 term coming from the extra field $u$.   
First, we prove the following. 
\begin{lemma}
Let $(A, u)$ be a D--T instanton of a vector bundle $E$ 
over $B \setminus \{ 0 \}$ with $\int_{B \setminus \{ 0 \}} 
|u|^2 dV < \infty $.  
Then, there exist constants $\varepsilon > 0$ and $C >0$ such that 
if $\int_{B} |F_{A}|^2  d V_{g} \leq \varepsilon $, then  
$|z|^4 \left| F_{A}  \right|^{2} (z) \leq 
C \varepsilon$ 
for $z \in B_{\frac{1}{2}}(0) \setminus \{ 0 \}$. 
\label{lem:gr2}
\end{lemma}

\begin{proof}
By the monotonicity formula \eqref{eq:monotone} and Proposition
 \ref{estu}, 
we have 
\begin{equation*}
\begin{split}
\frac{1}{|z|^2} \int_{B_{|z|} (z)} |F_{A}|^2 \, dV 
 &= \frac{1}{|z|^2} \int_{B_{|z|} (z)} m (A,u) \, d V  
 + \frac{1}{|z|^2} \int_{B_{|z|}(z)} | [ u , \bar{u}] |^2 \, dV\\ 
 &\leq \frac{1}{|z|^2} \int_{B_{2|z|} (0)} m (A,u) \, d V  
 + \frac{1}{|z|^2} \int_{B_{2|z|}(0)} | [ u , \bar{u}] |^2 \, dV\\
 &\leq \frac{C}{|z|^2} \int_{B} m (A,u) \, d V  
 + \frac{1}{|z|^2} \int_{B_{2|z|}(0)} | [ u , \bar{u}] |^2 \, dV\\
 &\leq C \varepsilon
\end{split}
\end{equation*}
for $z \in B_{\frac{1}{2}} \setminus \{ 0 \}$. 
Hence, 
using Proposition \ref{th:est}, we obtain 
\begin{equation*}
\left| F_{A} \right|^2 (z) 
\leq \frac{C_1^2}{|z|^6} \int_{B_{|z|} (z)} |F_{A}|^2 \, dV 
 + C_1^2 \varepsilon |z|^2   
\leq \frac{C \varepsilon}{|z|^4}   
\end{equation*} 
for sufficiently small $\varepsilon >0$.
\end{proof}

Then Lemma \ref{lem:gr2} enables us to use a particular choice of gauge, 
which is vital for removing the singularities. 
Let $U_{\ell} := \{ z \in B \, : \, 
2^{- \ell -1} \leq |z| \leq 2^{- \ell} \}$, and  
$S_{\ell} := \{ z \in B \, : \, 
|z| = 2^{- \ell} \}$  for each $\ell = 1 , 2, \dots$ 
so that $B_{\frac{1}{2}} \setminus \{ 0 \} = \cup_{\ell} U_{\ell}$. 
We recall the following from \cite{Uh2}.

\begin{definition}[\cite{Uh1} pp. 25]
A {\it broken Hodge gauge} for a connection $D$  
of a bundle $E$ over $B_{\frac{1}{2}} \setminus \{ 0\} = 
\cup_{\ell =1}^{\infty} U_{\ell}$ is a gauge 
related continuously to the original gauge in which 
$D= d + A$ and $A_{\ell}:= A|_{U_{\ell}}$ 
satisfies the following for all $\ell \geq 1$.  
(a) $d^{*} A_{\ell} = 0 $ in $U_{\ell}$, 
(b) $A_{\ell}^{\psi} |_{S_{\ell}} =
A_{\ell -1}^{\psi} |_{S_{\ell}}$, 
(c) $d^{*}_{\psi} A_{\ell}^{\psi} = 0$ on 
$S_{\ell}$ and $S_{\ell +1}$, and 
(d) $\int_{S_{\ell}} A_{\ell}^{r} 
= \int_{S_{\ell +1}} A_{\ell}^{r} = 0$, 
where the superscripts $r$ and $\psi$ indicate the radial and spherical
 components respectively. 
\end{definition}

Uhlenbeck proved the following existence result of the broken Hodge gauges on
 $\cup_{\ell} U_{\ell}$. 
\begin{theorem}[\cite{Uh1} Theorem 4.6]
There exists a constant $\gamma' > 0$ such that if $D$ is a smooth
 connection on $B_{\frac{1}{2}} \setminus \{0\}$, and the growth of the curvature
 satisfies $|x|^2 |F(x)| \leq \gamma \leq \gamma'$, 
then there exists a broken Hodge gauge in $B_{\frac{1}{2}} \setminus \{ 0 \}$ 
satisfying (e) 
$|A_{\ell} | (x) \leq C 2^{- \ell} || F_{A_{\ell}} ||_{L^{\infty}} 
\leq C 2^{\ell +1} \gamma $, 
and (f) 
$(\lambda - C \gamma ) \int_{U_{\ell}} |A_{\ell} |^2 dV \leq 
2^{-2 \ell} \int_{U_{\ell}} |F_{A_{\ell}} |^2 dV$, 
where $\lambda >0$ is a constant. 
\label{th:Uh}
\end{theorem}

Using this broken Hodge gauge, we deduce the following. 
\begin{lemma}
There exists a constant $\varepsilon> 0$ such that if 
$(A,u)$ is a D--T instanton of a bundle $E$ over $B \setminus \{ 0 \}$
 with $\int_{B} |F_{A} |^2 dV \leq \varepsilon$, 
then $|F_{A} |^2 (z) = o ( |z|^{-4 + \alpha} )$ for some $\alpha >0$ 
for all $z \in B_{\frac{1}{2}} (0)\setminus \{ 0 \}$. 
\label{lem:pa} 
\end{lemma}

\begin{proof}
By integration by parts, we get 
\begin{equation}
\begin{split}
\int_{U_{\ell}} |F_{A_{\ell}}|^2 dV_g 
 &= \int_{U_{\ell}} \langle  F_{A_{\ell}} , D A_{\ell} \rangle 
   - \frac{1}{2} \int_{U_{\ell}} \langle F_{A_{\ell}} , [ A_{\ell} ,
 A_{\ell}] \rangle \\
 &= \int_{U_{\ell}} \langle D^{*} F_{A_{\ell}} , A_{\ell} \rangle 
   - \frac{1}{2} \int_{U_{\ell}} \langle F_{A_{\ell}} , [ A_{\ell} ,
 A_{\ell}] \rangle \\
 & \qquad + \left( \int_{S_{\ell -1}} - \int_{S_{\ell}} \right) 
   \langle A_{\ell}^{\psi} \wedge F_{A_{\ell}}^{r \psi} \rangle .  \\
\end{split}
\label{eq:1}
\end{equation}
Firstly, we estimate the first and second terms 
in the last line of \eqref{eq:1}.  
By using the equations \eqref{DT1}, \eqref{DT2}, 
the H\"{a}lder inequality, 
 and (e) in Theorem \ref{th:Uh}, 
the first term in the last line of \eqref{eq:1} becomes 
\begin{equation}
 \begin{split}
\int_{U_{\ell}} \langle D^{*} F_{A_{\ell}} 
 ,  A_{\ell} \rangle 
 &= \int_{U_{\ell}} \langle * D ( \Lambda [ u , \bar{u} ])  , A_{\ell} \rangle
  \\ 
 &\leq \left( \int_{U_{\ell}} |D ( \Lambda [ u , \bar{u}])|^2 
dV_{g} \right)^{\frac{1}{2}}
  \left( \int_{U_{\ell}} | A_{\ell} |^2 dV_{g} \right)^{\frac{1}{2}} \\ 
 &\leq  \varepsilon C \left( \int_{U_{\ell}} | [u , \bar{u} ] |^2 \,
  dV_{g}  \right)^{\frac{1}{2}} ,\\   
 \end{split}
\label{eq:1-2}
\end{equation}
where $\Lambda = (\omega \wedge )^{*}$. 
For the second term in the last line of \eqref{eq:1}, 
we again use the H\"{a}lder inequality and (e) in Theorem
 \ref{th:Uh}.  
Then it becomes 
\begin{equation}
 \begin{split} 
 \frac{1}{2} \int_{U_{\ell}} 
   \langle F_{A_{\ell}} , [A_{\ell} , A_{\ell}] \rangle
   & \leq \frac{1}{2} \left( \int_{U_{\ell}} |F_{A_{\ell}} |^2
       \right)^{\frac{1}{2}} 
     \left( \int_{U_{\ell}} | A_{\ell} |^4 \right)^{\frac{1}{2}} \\
   & \leq \frac{1}{2} \left( \int_{U_{\ell}} |F_{A_{\ell}} |^2
       \right)^{\frac{1}{2}} 
     \left( C \varepsilon^2  
  \int_{U_{\ell}} | F_{A_{\ell}} |^2 \right)^{\frac{1}{2}}  \\ 
   & \leq C' \varepsilon \int_{U_{\ell}} |F_{A_{\ell}}|^2  . \\
 \end{split}
\label{eq:1-3}
\end{equation}
Thus, from \eqref{eq:1}, \eqref{eq:1-2}, and \eqref{eq:1-3}, 
we get 
\begin{equation}
 (1  - \varepsilon C') \int_{U_{\ell}} |F_{A}|^2  
    - \varepsilon C \int_{U_{\ell}} | [u, \bar{u} ]|^2 
 \leq \left( \int_{S_{\ell -1}} - \int_{S_{\ell}} \right) 
    \langle A^{\psi}_{1} , F_{A_1}^{r \psi} \rangle d \sigma . 
\label{eq:5}
\end{equation}
Hence, taking $\varepsilon$ small,  and summing up \eqref{eq:5} in
 $\ell$, 
we obtain 
\begin{equation}
 \int_{B_{\frac{1}{2}}} |F_{A}|^2  
    - \int_{B_{\frac{1}{2}}} | [u, \bar{u} ]|^2 
 \leq 2  \int_{S_{1}} 
    \langle A^{\psi}_{1} , F_{A_1}^{r \psi} \rangle d \sigma . 
\label{eq:1-4}
\end{equation}
We next estimate the right-hand side of \eqref{eq:1-4}. 
Firstly, we have  
$$ \int_{S_{1}}
    \langle A^{\psi}_{1} , F_{A_1}^{r \psi} \rangle d\sigma 
\leq K \int_{S_{1}} | A^{\psi}_{1} |^2 \, d \sigma  
 + \frac{1}{K} \int_{S_1} |F_{A_1}^{r \psi} |^2 \, d \sigma , $$
where $K>0$ is a constant determined later. 
By using an estimate in \cite[Th.~2.5]{Uh1}, 
the first term of the right-hand side of the above inequality is
 bounded as follows.  
$$ \int_{S_1} |A_{1}^{\psi} |^2  \, d \sigma 
 \leq C \int_{S_1} |F_{A_1}^{\psi \psi} |^2 \, d \sigma \leq C  \int_{S_1} 
 | F_{A_1} |^2 \, d \sigma . $$
Hence, by the rescaling $y = z /r$, we get 
\begin{equation}
 \frac{1}{r^2} \int_{B_{r}} m (A, u) 
\leq \frac{C K}{r} \int_{\partial B_{r}} |F_{A}|^2 d \sigma   
 + \frac{C}{K r} \int_{\partial B_{r}} | F_{A}^{r \psi}|^2 d \sigma . 
\label{eq:2}
\end{equation}
We then put $M(r) := e^{a r^2} r^{-2} \int_{B_{r}} m (A,u) \, dV_{g}$. 
Multiplying $e^{a \rho^2}$ to \eqref{eq:2}, 
and integrating it from $\rho/2$ to $\rho$, 
we obtain  
\begin{equation}
\begin{split}
\int_{\frac{\rho}{2}}^{\rho} 
M(r) \, dr 
&\leq C K e^{a \rho^2} \rho^{-1} \int_{B_{\rho}} |F_{A}|^2 \, dV_{g} \\
 &\qquad \qquad + C K^{-1} \int_{\frac{\rho}{2}}^{\rho} 
 e^{a r^2} r^{-1} \int_{\partial B_{r}} 
   |F_{A}^{r \psi} |^2 \, d\sigma d r . \\
\end{split}
\label{eq:3}
\end{equation}
From the monotonicity formula \eqref{eq:monotone}, we deduce that 
the left-hand side of \eqref{eq:3} is bounded by 
$\frac{\rho}{2} M\left( \frac{\rho}{2}\right)$ from below. 
On the other hand, the second term of the right-hand side of \eqref{eq:3} is estimated as
follows. 
\begin{equation*}
\begin{split}
C K^{-1} \int_{\frac{\rho}{2}}^{\rho} 
 & e^{a r^2} r^{-1} \int_{\partial B_{r}} 
   |F_{A}^{r \psi} |^2 \, d\sigma d r \\
& \leq C K^{-1} \rho \left( e^{a \rho^2} \rho^{-2} 
 \int_{B_{\rho}} |F_{A} |^2 \, dV_{g} - 2 e^{a \rho^2 /4 } \rho^{-2}
 \int_{B_{\frac{\rho}{2}}}
 |F_{A} |^2 \, dV_{g} \right) \\
&\leq CK^{-1} \rho 
\left( M (\rho) -  M ( \rho /2) \right) 
 + C K^{-1} e^{a \rho^2} \rho \left( 
\rho^{-2} \int_{B_{\rho}} | [ u , \bar{u}] |^2 \right) . \\  
\end{split}
\end{equation*}
Hence we get 
$$ \left( \frac{1}{2} + C K^{-1} 
\right) M ( \rho / 2) 
\leq \left( C K + C K^{-1} \right) 
M ( \rho) + C' \varepsilon^2 
e^{a \rho^2} \rho^{4} . $$
Thus, by taking $K$ small, we obtain 
$ \zeta M ( \rho /2 ) \leq 
M (\rho ) + C' \varepsilon^2 e^{a \rho^2} \rho^{4} $  
for some $\zeta > 1$. 
Hence, by iteration, we obtain   
$$ M ( 2^{-\ell}) \leq \left( 2^{-\ell} \right)^{\log_{2} \zeta} M ( 1/ 2) 
   + C'' \varepsilon^2 \left( 2^{- \ell} \right)^4 . $$
Therefore,  
\begin{equation}
M ( \rho ) \leq C \rho^{\alpha} \int_{B} m(A,u) \, dV_{g}  
+ C'' \varepsilon^2 \rho^{4}, 
\label{eq:4}
\end{equation}
where $\alpha = \log_{2} \zeta$.

Hence, form Proposition \ref{th:est} and  \eqref{eq:4}, we get 
\begin{equation*}
\begin{split}
|F_{A}|^2  (z)
& \leq \frac{C}{|z|^4} \left( \frac{1}{|z|^2} \int_{B_{2|z|}} 
|F_{A}|^2 \, dV_{g} \right) 
 + C \varepsilon^{2} |z|^2 \\
& \leq  C' |z|^{-4 + \alpha} \int_{B} |F_{A}|^2 \, dV_{g} 
   + C'' \varepsilon^2 +  C \varepsilon^2 |z|^2 .
\end{split} 
\end{equation*}
Thus, Lemma \ref{lem:pa} is proved. 
\end{proof}

From Lemma \ref{lem:pa}, we deduce $F_{A} \in L^{p}$ for some $p >3$. 
Hence, a theorem by Uhlenbeck \cite[Th.~2.1]{Uh2} (see also 
 \cite[Chap.~6]{KW}) tells us that there is a gauge
 transformation $\sigma \in L^{p}_{2}$ such that $\sigma (A,u)$ smoothly
 extends over $B$. 
As $L_{2}^{p} \subset C^0$ for $p>3$, and $A$ and $\sigma (A)$ are
 smooth on $B \setminus \{ 0 \}$, 
we then realise that $\sigma$ is also smooth on $B \setminus \{ 0 \}$. 
\end{proof}


\begin{flushleft}
E-mail: tanaka.yuuji@math.nagoya-u.ac.jp
\end{flushleft}

\end{document}